\documentclass[12pt,a4paper,reqno] {amsart}

\usepackage{eqnarray,amsmath}
\usepackage{amssymb}
\usepackage{mathrsfs}
\usepackage{calc}
\usepackage{etoolbox}

\patchcmd{\section}{\scshape}{\bfseries}{}{}
\makeatletter
\renewcommand{\@secnumfont}{\bfseries}
\makeatother

\newcommand{\mascD}{\mathscr D}
\newcommand{\maclD}{\mathcal D}
\newcommand{\mascS}{\mathscr S}
\newcommand{\maclS}{\mathcal S}

\newcommand{\op}{\operatorname{Op}}

\numberwithin{equation}{section}          

\newtheorem{thm}{Theorem}
\numberwithin{thm}{section}

\newcommand{\rubrik}{}
\newtheorem{prop}[thm]{Proposition}

\newtheorem{lemma}[thm]{Lemma}

\theoremstyle{definition}

\newtheorem{defn}[thm]{Definition}

\newcommand{\rubrikdef}{}

\theoremstyle{remark}

\begin{document}

\title[Gevrey kernels to positive operators]
{Boundedness for Gevrey and Gelfand-Shilov kernels to
\\
positive operators}

\author{Yuanyuan Chen}

\address{Department of Mathematics,
Linn{\ae}us University, V{\"a}xj{\"o}, Sweden}

\email{yuanyuan.chen@lnu.se}

\author{Joachim Toft}

\address{Department of Mathematics,
Linn{\ae}us University, V{\"a}xj{\"o}, Sweden}

\email{joachim.toft@lnu.se}

%


\keywords{Positivity, twisted convolutions, ultra-distributions,
Weyl quantization, kernels}

\begin{abstract}
We study properties of positive operators 
on Gelfand-Shilov spaces, and distributions
which are positive with respect to non-commutative
convolutions. We prove that boundedness of kernels 
$K \in \maclD_s^{\prime}$ to positive operators, 
are completely determined by the 
behaviour of $K$ alone the diagonal. We also prove that 
positive elements $a$ in $\mascS^{\prime}$ with respect 
to twisted convolutions, having Gevrey class property 
of order $s\geq 1/2$ at the origin,
then $a$ belongs to the Gelfand-Shilov space $\maclS_s$.
\end{abstract}

\maketitle

\par

\section{Introduction}\label{sec0}

\par

In this paper we deduce boundedness properties for 
kernels of positive operators and elements with 
respect to non-commutative convolutions.
More precisely, consider a Roumieu distribution
(i.e. an ultra-distributions of Roumieu type), which
at the same time is a kernel to a positive operator.
Roughly speaking, we prove that the kernel is a 
Gelfand-Shilov distribution of certain degree, if 
and only if its restriction to the diagonal is also 
a Gelfand-Shilov distributions of the same degree.

\par

A consequence of this result is that a Roumieu 
distribution which is positive with respect to a 
non-commutative convolution algebra, 
belongs to corrsponding space of 
Gelfand-Shilov distributions.

\par

We remark that the usual convolution as well as
the twisted convolution are special cases of these
non-commutative convolutions. 

\par

For the twisted convolution we perform further 
investigations when the Roumieu distributions
possess stronger regularity. More precisely, 
if a Roumieu or Gelfand-Shilov distribution is 
positive with respect to the twisted convolutions,
and is of Gevery class of certain degree near 
the origin, then we prove that the distribution
is a Gelfand-Shilov function of the same 
degree.

\par

These results are analogous to results in \cite{TJ},
where similar properties were deduced after
the spaces of Roumieu distributions, Gelfand-Shilov
distributions and Gelfand-Shilov functions are
replaced by corresponding test function and 
distribution spaces in the standard distribution
theory.

\par

In order to describe our results in more details, 
we recall the definitions of positive operators and
elements which are positive with respect
to non-commutative convolutions.

Let $\mathscr B$ be a topological vector space with 
complex dual $\mathscr B^{\prime}$, and let $T$ be 
a linear and continuous operator from $\mathscr B$
to $\mathscr B^{\prime}$. Then $T$ is called positive
(semi-definite), and is written $T\geq 0$, whenever
$(T\varphi,\varphi)\geq 0$ for every $\varphi\in 
\mathscr B$. In our situation $\mathscr B$ is 
usually a Gelfand-Shilov space or a Gevrey class on
$\mathbf R^{d}$, and $\mathscr B
^{\prime}$ corresponding distribution space.
However we remark that we may also 
as in \cite{TJ} consider the 
case when $\mathscr B$ is the set of Schwartz 
functions and $\mathscr B^{\prime}$ the set of 
tempered distributions on
$\mathbf R^{d}$.

In any of these situations, the Schwartz kernel theorem
is valid in the sense that for any linear and continuous
operator from $\mathscr B$ to $\mathscr B^{\prime}$,
there is a distribution $K=K_T\in \mathscr B^{\prime}
\otimes \mathscr B^{\prime}$ such that 
\begin{equation}\label{Ker}
(T\varphi, \psi)=(K, \psi\otimes \bar{\varphi}), 
\quad \quad \text{for every}
\quad \varphi, \psi\in \mathscr B.
\end{equation}

We are especially interested in properties possessing 
by elements $K\in \mathscr B^{\prime}\otimes \mathscr 
B^{\prime}$ such that $T=T_K$, defined by \eqref
{Ker}, are positive as operators. Note here that if 
$\mathscr B$ is a Gelfand-Shilov space or the set of 
Schwartz functions on $\mathbf R^d$, then $\mathscr B
\otimes \mathscr B$ is the corresponding space of 
functions defined on $\mathbf R^{2d}$, and 
$\mathscr B^{\prime} \otimes \mathscr B^{\prime}$
the corresponding distribution space on $\mathbf R^{2d}$.

An importment special case of $T$ concerns 
non-commutative convolution operators of the form
$\varphi \mapsto a\ast_B\varphi$, defined by
$$
(a\ast_B\varphi )(x)\equiv \int \!
a(x-y)\varphi (y)B(x,y) \,dy .
$$
In this situation we are for example interested in
extensions of Bochner-Schwartz theorem 
concerning positive elements in 
non-commutative
convolutions.

An important non-commutative convolution is the 
twisted convolution, which is obtained by choosing 
$B$ above as an appropriate complex
Gaussian. More precisely, let $X=(x,\xi)\in \mathbf 
R^{2d}$ and $Y=(y,\eta)\in \mathbf R^{2d}$, and 
let $\sigma$ be the standard symplectic form
defined by 
$ \sigma(X,Y)=\langle y,\xi\rangle-\langle x,
\eta\rangle $.
Then the twisted convolution is defined by 
 $$
(a\ast_{\sigma}b)(X)\equiv (2/\pi)^{d/2}\int 
a(X-Y)b(Y)e^{2i\sigma(X,Y)} \,dY .
$$

\par

Positivity in the twisted convolution is closely related to 
positivity in operator theory,
especially the Weyl calculus. 
More precisely,
if $a\in \mathscr B\otimes \mathscr B$, 
then $Aa$ is the operator with Schwartz 
kernel giving by 
\begin{equation}\label{Aaopera}
(Aa)(x,y)=(2\pi)^{-d/2}\int a((y-x)/2,\xi)
e^{-i\langle x+y,\xi \rangle}\,d\xi.
\end{equation}
(Cf. \cite{TJ}.)
Here and in what follows we identify operators with
their kernels. The operator $A$ in \eqref{Aaopera} is 
continuous on $\mathscr B\otimes \mathscr B$, and 
extends uniquely to a continuous map on 
$\mathscr B^{\prime} \otimes \mathscr B^{\prime}$.

\par

The main relationship on positivity is that $Aa\geq 0$,
if and only if $a$ belongs to $(\mathscr B^{\prime}
\otimes \mathscr B^{\prime})_+$, the set of all 
elements in $\mathscr B^{\prime}
\otimes \mathscr B^{\prime}$
which are positive with respect
to the twisted convolution.
 (Cf. Proposition 1.10 in \cite{TJ}.)

\par

There are also strong links between positivity in
operator theory (or equivalently elements in 
$(\mathscr B^{\prime}
\otimes \mathscr B^{\prime})_+$) and positive Weyl
operators. In fact, for 
$a\in \mascS(\mathbf R^{2d})$, the Weyl quantization
$\op^{w}(a)$, is the operator from $ \mascS(\mathbf 
R^{d})$ to $ \mascS^{\prime}(\mathbf R^{d})$, defined
by the formula
$$
(\op^{w}(a)f)(x)=
(2\pi)^{-d}\iint a((x+y)/2,\xi)f(y)
e^{i\langle x-y,\xi \rangle} dy \,d\xi.
$$
The integral kernel of $\op^{w}(a)$ is equal to 
\begin{equation}\label{opkernel}
(x,y)\mapsto (2\pi)^{-d/2}(Aa)(-x,y).
\end{equation}
For arbitrary $a\in \mathscr B^{\prime}
\otimes \mathscr B^{\prime}$, $\op^{w}(a)$
is defined as the operator with kernel given by
\eqref{opkernel}.

\par

By straight-forward computations we get
$$
\op^{w}(a)=(2\pi)^{-d/2}A(\mathscr F_{\sigma}a),
$$
where $\mathscr F_{\sigma}$ is the symplectic
Fourier transform on $\mathscr B^{\prime}
\otimes \mathscr B^{\prime}$, which takes 
the form
$$
(\mathscr F_{\sigma}a)(X)=\widehat{a}(X)
=\pi^{-d}\int a(Y)e^{2i\sigma \langle X,Y\rangle}dY,
$$
when $a\in \mascS(\mathbf R^{2d})$. Consequently,
from these identities it follows that
$\op^{w}(a) \geq 0$, if and only if 
$(\mathscr F_{\sigma}a)\in (\mathscr B^{\prime}
\otimes \mathscr B^{\prime})_+$.

\par

In the paper we begin to study kernels in Roumieu 
distribution spaces $\maclD_{s}^{\prime}$, $s>1$,
whose corresponding operators are positive
semi-definite. In fact, if $K$ is such kind of kernel 
whose restriction near the diagonal belongs to 
corresponding Gelfand-Shilov distribution space  
$\maclS_{s}^{\prime}$, then we prove that $K$
belongs to $\maclS_{s}^{\prime}$. 

\par

By choosing $K$ such that $T_K$ agree with the 
convolution operator $\varphi \mapsto 
a\ast_B\varphi$, for some $a\in \maclD_s^{\prime}$,
it follows that $a\in \maclS_s^{\prime}$ when the 
operator is positive. 
In particular, this holds for elements which are positive
with respect to twisted convolution.

\par

We also perform investigations on elements in
$(\mathscr B^{\prime}
\otimes \mathscr B^{\prime})_+$
which satisfy certain smoothness condition at
origin. More precisely, if $a\in (\mathscr B^{\prime}
\otimes \mathscr B^{\prime})_+$ is Gevrey regular 
of order $s$, then we prove that $a$ belongs to 
$\maclS_s(\mathbf R^{2d})$. This result is 
analogous to Theorem 3.13 in \cite{TJ}, which
asserts that for any $a\in (\mathscr B^{\prime}
\otimes \mathscr B^{\prime})_+$ which is 
smooth near origin, belongs to $\mascS
(\mathbf R^{2d})$.

\section{Preliminaries}\label{sec1}

\par 

In this section we recall some basic results which are needed.
In the first part we recall some facts about Gelfand-Shilov
spaces.

\par

In the following we let $\alpha $, $\beta$, $\gamma$ and
$\delta$ denote multi-indices.

\par

Let $s\geq 1/2$ and $h>0$. Then $\maclS_{s,h}({\mathbf R^d})$ 
is the set of all $\varphi \in C^{\infty}({\mathbf R^d})$ 
such that the norm 
$$
\| \varphi \|_{\maclS_{s,h}}\equiv
\sup_{\alpha,\beta \in \mathbf N^d}\sup_{x\in \mathbf R^d}
\frac{|x^{\alpha}D^{\beta}\varphi(x)|}
{(\alpha! \beta!)^sh^{|\alpha+\beta|}} 
$$
is finite for every multi-indices $\alpha$ and $\beta$.
Then the Gelfand-Shilov space
$\maclS_s({\mathbf R^d})$ is given by 
$$
\maclS_s({\mathbf R^d}) =\operatorname{ind}  
\lim_{h\rightarrow \infty}
\maclS_{s,h}({\mathbf R^d}).
$$
Its dual space is denoted by 
$\maclS_s^{\prime}({\mathbf R^d})$,
which is the Gelfand-Shilov distribution space
of order $s$.

\par

Let $s\geq 1/2$, $h>0$ and let $\Omega$ be an 
open set in $\mathbf R^d$.
For a given compact set $K\subset \Omega$, 
$\maclD_{s,h}(K)$
is the set of all  $\varphi \in
C^{\infty}({\mathbf R^d})$ such that 
$\operatorname{supp} \varphi
\subseteq K$ and the norm
$$
\| \varphi \|_{\maclD_{s,h}}\equiv
\sup_{\beta \in \mathbf N^d}\sup_{x\in K}
\frac{|D^{\beta}\varphi(x)|}
{(\beta!)^sh^{|\beta|}}
$$
is finite for the multi-index $\beta$.
Let $(K_n)_n$ be a sequence of compact sets such that 
$K_n\subset \subset K_{n+1}$, and $\bigcup K_n=\Omega$.
Then the space $\maclD_s(\Omega)$ is given by 
$$
\maclD_s(\Omega)=\operatorname {ind}
\lim_{n\rightarrow \infty}
(\operatorname {ind}
\lim_{h\rightarrow \infty} \maclD_{s,h}(K_n)).
$$
Its dual space is $\maclD_s^{\prime}(\Omega)$,
which is the ultra-distribution space of Roumieu type 
of oreder $s$.
We remark that $\maclD_s$ is equivalent to 
Gevery class $G^s$ for $s>1$.

\par 

We recall that differentiations of Gevrey or 
Gelfand-Shilov distributions are defined in 
the usual way, giving that most of the usual
rules hold. Especially it follows from Leibniz
rule that 
\begin{equation} \label{eq.leib}
D^{\alpha}x^{\beta} f(x)= 
\sum _{\alpha_0\le \alpha,\beta }
(-i) ^{\alpha _0} {\alpha \choose {\alpha _0}}
{\beta \choose \alpha _0} 
\alpha_0! x^{\beta-\alpha_0}
D^{\alpha-\alpha_0}f(x),
\end{equation}
for admissible distribution $f$. 
Furthermore, by applying the Fourier
transfor to this formula we get
\begin{equation} \label{eq.exchangorder}
x^{\alpha}D^{\beta}f(x) = 
\sum_{\alpha_0 \le \alpha,\beta}
(-1)^{\gamma}
{\alpha \choose \alpha_0}
{\beta \choose \alpha_0}\alpha!
D^{\beta-\alpha_0}(x^{\alpha-\alpha_0}f(x)).
\end{equation}

\par

Since we are especially interested in the 
behaviour of corrsponding Schwartz
kernels, the following result is important 
to us. The proof is omited since it
can be found in \cite{CP}.

\begin{thm}
Let $\Omega_j \subset \mathbf R^{d_j}$ be open set, 
where $j=1,2$. 
Then the following are true. 
\begin{enumerate}
\item If $s>1$ and $T$ is a linear and continuous operator from 
$\maclD_s(\Omega_1)$ to $\maclD_s^{\prime}(\Omega_2)$, 
then there is a unique ultradistribution
$K=K_T\in \maclD_s^{\prime}(\Omega_2 \times \Omega_1)$
such that 
\begin{equation}\label{kernelD}
(T\varphi_1,\varphi_2)=(K,\varphi_2 \otimes \overline \varphi_1), 
\quad \varphi_1 \in \maclD_s(\Omega_1), 
\varphi_2 \in \maclD_s(\Omega_2).
\end{equation}
Conversely, if $K=K_T\in \maclD_s^{\prime}
(\Omega_2 \times \Omega_1)$ and 
$T$ is defined by \eqref{kernelD}, then $T$ is a 
linear and continuous operator from 
$\maclD_s(\Omega_1)$ to $\maclD_s^{\prime}(\Omega_2)$.

\item If $s\geq 1/2$ and $T$ is a 
linear and continous operator from 
$\maclS_s(\mathbf R^{d_1})$ to 
$\maclS_s^{\prime}(\mathbf R^{d_2})$, 
then there is a unique tempered ultradistribution
$K=K_T \in \maclS_s^{\prime}
(\mathbf R^{d_2} \times \mathbf R^{d_1})$ 
such that 
\begin{equation}\label{kernelS}
(T\varphi_1,\varphi_2)=(K,\varphi_2 
\otimes \overline \varphi_1), 
\quad \varphi_1 \in \maclS_s(\mathbf R^{d_1}),
\varphi_2 \in \maclS_s(\mathbf R^{d_2}).
\end{equation}
Conversely, if $K=K_T
\in \maclS_s^{\prime}(\mathbf R^{d_2} 
\times \mathbf R^{d_1})$ and $T$ is defined
by \eqref{kernelS} , then $T$ is a 
linear and continuous operator from 
$\maclS_s(\mathbf R^{d_1})$ to 
$\maclS_s^{\prime}(\mathbf R^{d_2})$.
\end{enumerate}
\end{thm}

If $K$ is given, then $T_K=T$ is 
defined by \eqref{kernelD}. 

\par

\vspace{0.3cm}

\par

Recall that if $s \geq 1/2$, and $T$ is a 
linear and continuous operator
from $\maclS_s(\mathbf R^d)$ to 
$\maclS_s^{\prime}(\mathbf R^d)$,
then $T$ is called positive semi-definite if 
\begin{equation} \label{posi semi}
(T\varphi,\varphi)_{L^2} \geq 0,
\end{equation}
for every $\varphi\in \maclS_s
(\mathbf R^d)$, and then we write $T\geq 0$.
Furthermore, if more restrictive $s>1$, 
$\Omega \subset \mathbf R^d$ is an open set, 
and $T$ is a linear and continous operator 
from $\maclD_s(\Omega)$ to 
$\maclD_s^{\prime}(\Omega)$, 
then $T$ is still called positive semi-definite 
when \eqref{posi semi} holds for every 
$\varphi\in \maclD_s(\Omega)$.

Since $\maclD_s({\mathbf R^d})$ is dense 
in $\maclS_s({\mathbf R^d})$ when
$s>1$, it follows that if $T$ from 
$\maclD_s({\mathbf R^d})$ to 
$\maclD_s^{\prime}({\mathbf R^d})$ 
is positive semi-definite 
and extendable to a continuous
map from $\maclS_s({\mathbf R^d})$ to 
$\maclS_s^{\prime}({\mathbf R^d})$,
then this extention is unique and $T$ is 
positive semi-definte as an operator from 
$\maclS_s({\mathbf R^d})$ to 
$\maclS_s^{\prime}({\mathbf R^d})$.

\par

\vspace{0.3cm}

\par

We have now the 
following definition. 

\begin{defn}
Let $\Omega \subset \mathbf R^d$ be open.
\begin{enumerate}
\item If $s>1$, then $\maclD_{0,s}
^{\prime}(\Omega \times \Omega)$  
consists of all $K\in \maclD_s^{\prime}
(\Omega \times \Omega)$ 
such that $T_K$ is a positive semi-definite 
operator from $\maclD_s({\Omega})$
to $\maclD_s^{\prime}({\Omega})$.
\item If $s\geq 1/2$, then $\maclS_{0,s}
^{\prime}(\mathbf R^d \times \mathbf R^d)$  
consists of all $K\in \maclS_s^{\prime}
(\mathbf R^d \times \mathbf R^d)$ 
such that $T_K$ is a positive semi-definite 
operator from $\maclS_s(\mathbf R^d)$
to $\maclS_s^{\prime}(\mathbf R^d)$.
\end{enumerate}
\end{defn}

\par 

We shall also consider distributions which
are positive with respect to a non-commutative 
convolution. 

Let $s>1$, 
and $B\in C_s(\mathbf R^{2d})$ 
such that for every 
$\varepsilon>0$, it holds
\begin{equation}\label{ineqB}
\sup_{\alpha}\sup_{x,y} e^{-\varepsilon(|x|^{1/s}+|y|^{1/s})}
\left( \frac{|D^{\alpha}(B(x,y))^{-1}|}
{(\alpha!)^sh^{|\alpha|}} +\frac{|D^{\alpha}B(x,y)|}
{(\alpha!)^sh^{|\alpha|}}\right)<\infty,
\end{equation}
for some $h>0$.

\par

\begin{lemma}\label{BCon}

Let $B\in C_s(\mathbf R^{2d})$ satisfying \eqref {ineqB}.
Then the following are ture.
\begin{enumerate}
\item $\Phi \in \maclS_s(\mathbf R^{2d})$ if and only if 
$B\cdot \Phi \in \maclS_s(\mathbf R^{2d})$.
\item $\Phi \in \maclS_s^{\prime}
(\mathbf R^{2d})$ if and only if 
$B\cdot \Phi \in \maclS_s^{\prime}(\mathbf R^{2d})$.
\end{enumerate}
\end{lemma}

\par

\begin{proof}
The result follows from Theorem A in \cite{RW}.
\end{proof}

\par

Let $a\in \maclD_s^{\prime}(\mathbf R^d)$ such that
$$
(a\ast_B\varphi)(x) \equiv \langle a(x-\cdot),B(x,\cdot) 
\varphi\rangle,
$$
when $\varphi \in  \maclD_s(\mathbf R^d)$.
Then the kernel of the map $\varphi \mapsto u\ast_B\varphi$
is given by $K(x,y)=a(x-y)B(x,y)$. We note that $K(x,y)
\in \maclD_s^{\prime}(\mathbf R^{2d})$

\par

Let $\maclD_{B,s,+}^{\prime}(\mathbf R^d)$ be the set of all
$a\in \maclD_s^{\prime}(\mathbf R^d)$ such that the map 
$\varphi \mapsto a\ast_B\varphi$ is positive.
Also let $\maclS_{B,s,+}^{\prime}(\mathbf R^d)$ 
be the set of all $a\in \maclS_s^{\prime}(\mathbf R^d)$
such that the map $\varphi \mapsto a\ast_B\varphi$ is positive.

\par

We consider elements $a$ in Gelfand-Shilov 
classes of distributions which
are positive with respect to the twisted 
convolution. That is, $a$ should fulfill
\begin{equation} \label {eq.positive}
(a \ast_\sigma \varphi,\varphi) \geq 0,
\end{equation}
for every $\varphi$.
For this reason we make the following definition. 

\par

\begin{defn}
Let $s\geq 1/2$.
\begin{enumerate}
\item $\maclS _{s,+}^{\prime}(\mathbf R^{2d})$ 
is the set of all $a\in \maclS
_s^{\prime}(\mathbf R^{2d})$ such that 
\eqref{eq.positive} holds for every $\varphi
\in \maclS _s(\mathbf R^{2d})$.

\item If in addition $s>1$, then 
$\maclD _{s,+}^{\prime}(\mathbf R^{2d})$ is the
set of all $a\in \maclD _s^{\prime}
(\mathbf R^{2d})$ such that 
\eqref{eq.positive} holds for every 
$\varphi \in \maclD _s(\mathbf R^{2d})$.

\item $\mascS _+^{\prime}(\mathbf R^{2d})$ 
is the set of all $a\in \mascS
^{\prime}(\mathbf R^{2d})$ such that 
\eqref{eq.positive} holds for every
$\varphi \in \mascS(\mathbf R^{2d})$.

\item $\mascD _+^{\prime}(\mathbf R^{2d})$ 
is the set of all $a\in \mascD
^{\prime}(\mathbf R^{2d})$ such that  
\eqref{eq.positive} holds for every
$\varphi \in C_0^{\infty}(\mathbf R^{2d})$.

\item The set $C_+(\mathbf R^{2d})$ 
consists of all $a\in C(\mathbf R^{2d})$ 
such that 
$$
\sum_{j,k}a(X_j-X_k)e^{2i\sigma( X_j,X_k)} 
c_j\overline{c_k} \geq 0,
$$
for every finite sets
$$
\{ X_1,X_2,\ldots,X_N \} \subseteq \mathbf R^{2d}
\quad \text{and}\quad 
\{ c_1,c_2,\ldots, c_N \} \subseteq \mathbf C.
$$
\end{enumerate}
\end{defn}

\par 

The following result can be found in 
the Theorem 2.6, Proposition 3.2 
and Theorem 3.13 in \cite{TJ}.
Later on we shall prove an analogue 
in the frame-work Gelfand-Shilov 
spaces. 

\par

\begin{prop}\label{DCS}
Let $\Omega \subseteq  \mathbf R^{2d}$ 
be a neighborhood of the origin. 
Then the following are true.
\begin{enumerate}
\item $\mascD_+^{\prime}(\mathbf R^{2d})=
\mascS_+^{\prime}(\mathbf R^{2d})$.

\item $C_+(\mathbf R^{2d})\subseteq 
\mascS_+^{\prime}(\mathbf R^{2d}) 
\cap L^2(\mathbf R^{2d})\cap  L^{\infty}(\mathbf R^{2d})
\cap \mathscr F L^{\infty}(\mathbf R^{2d})$.

\item $\mascS_+^{\prime}(\mathbf R^{2d}) 
\cap C(\Omega)=C_+(\mathbf R^{2d})$.

\item $\mascS_+^{\prime}(\mathbf R^{2d}) 
\cap C^{\infty}(\Omega)=
C_+(\mathbf R^{2d})\cap \mascS(\mathbf R^{2d})$.
\end{enumerate}
\end{prop}

\par

\medspace

\par

By Proposition 1.5 in \cite{TJ} and the 
definitions we have the following.

\par

\begin{prop}\label{trace}
Let $s\geq 1/2$, $a\in \maclS_s^{\prime}
(\mathbf R^{2d})$ be such that $Aa$ is 
a trace-class operator on $L^2({\mathbf R^d})$. 
Then $a\in L^{\infty}({\mathbf R^{2d}})$,
and 
$$
\|a\|_ {L^{\infty}} \leq (2/\pi)^{d/2}
\|Aa\|_{\operatorname{Tr}}.
$$
\end{prop}

\par

\section{Gelfand-Shilov properties 
for kernels to positive operators}\label{sec2}

\par 

In this section, we study the kernel 
of a positive semi-definite 
operator. If the kernel belongs to $\maclS_{s}^{\prime} $
along the diagonal, then it belongs to $\maclS_{s}^{\prime} $
everywhere for $s>1$. As a application we prove that 
$\maclD_{B,s,+}^{\prime}(\mathbf R^d)=
\maclS_{B,s,+}^{\prime}(\mathbf R^d)$.

\par

\begin{thm}\label{thmkernel}
Let $s>1$. Assume that $K\in \maclD_{0,s}^{\prime}
(\mathbf R^{2d})$,
and $K_\chi(x,y) \in \maclS_{s}^{\prime} (\mathbf R^{2d} )$,
where $K_\chi(x,y)=\chi(x-y)K(x,y)$, for some 
$\chi \in \maclD_s(\mathbf R^d)$
which satisfies $\chi(0) \ne 0$. 
Then $K\in \maclS_{0,s}^{\prime}
(\mathbf R^{2d})$.
\end{thm}

\par 

We need the following preparations for 
the proof of Theorem $\ref{thmkernel}$. 

\par

\vspace{0.5cm}

\par

Here let $s\geq 1/2$, $h$, $h_1$ and $h_2>0$,
and let $\psi \in \maclS_s(\mathbf R^{2d})$.
The semi-norms $\|\psi\|^{(1)}_{\maclS_{s,h}} $
and $\|\psi\|^{(2)}_{\maclS_{s,h_1,h_2}}$
are giving by
\begin{eqnarray*}
\|\psi\|^{(1)}_{\maclS_{s,h}} &\equiv& 
\sup_{\alpha_1,\alpha_2,\beta_1,\beta_2}
\sup_{x,y}\frac{|x^{\alpha_1}y^{\alpha_2}
D_x^{\beta_1}D_y^{\beta_2}\psi(x,y)|}
{(\alpha_1!\alpha_2!\beta_1!\beta_2!)^s
h^{|\alpha_1+\alpha_2+\beta_1+\beta_2|}},
\\
\|\psi\|^{(2)}_{\maclS_{s,h_1,h_2}} &\equiv&  
\sup_{\alpha,\beta_1,\beta_2}
\sup_{x,y}\frac{e^{h_2|y|^{1/s}}|
x^{\alpha}D_x^{\beta_1}D_y^{\beta_2}\psi(x,y)|}
{(\alpha!\beta_1!\beta_2!)^s
h_1^{|\alpha+\beta_1+\beta_2|}}.
\end{eqnarray*}

\par

\begin{lemma}\label{twoseminorm}
Let $s\geq1/2$,  and let $\psi_0 \in 
C^{\infty}(\mathbf R^{2d})$.
Then the following are true:
\begin{enumerate}
\item For every $h>0$, there are 
$C$, $h_1$, $h_2>0$ such that 
$$
\|\psi_0\|^{(1)}_{\maclS_{s,h}} 
\leq C\|\psi_0\|^{(2)}_{\maclS_{s,h_1,h_2}}.
$$  
\item For every $h_1$, $h_2>0$, 
there are $C$, $h>0$ such that 
$$
\|\psi_0\|^{(2)}_{\maclS_{s,h_1,h_2}} 
\leq C\|\psi_0\|^{(1)}_{\maclS_{s,h}} .
$$
\item Let $\psi_1$ and $\psi_2$ be given by 
\begin{equation}\label{psi1psi2}
\psi_1(x,y)=\psi_0(x,x-y), 
\quad \text{and} \quad \psi_2(x,y)=\psi_0(x+y,y).
\end{equation}
Then 
$$
\|\psi_j\|^{(1)}_{\maclS_{s,h}} 
\leq \|\psi_0\|^{(1)}_{\maclS_{s,h/4}},
\quad \text{and} \quad
\|\psi_0\|^{(1)}_{\maclS_{s,h}} 
\leq \|\psi_j\|^{(1)}_{\maclS_{s,h/4}} ,
$$
for $j=0,1,2$.
\end{enumerate}
\end{lemma}

\par

\begin{proof}
The assertions $(1)$ and $(2)$ follow
from Corollary 2.5 in \cite{CCK} and its proof.
The assertion $(3)$ follows by straight-forward
elaborations with the semi-norm 
$\|\psi_j\|^{(1)}_{\maclS_{s,h}} $, for 
$j=0,1,2$. In order to be self-contained 
we here give the proof. 

We only prove the result in the first inequality
in $(3)$ for $j=2$. We have 
\begin{multline*}
\frac{|(x-y)^{\alpha_1}y^{\alpha_2}\psi_0(x,y)|}
{(\alpha_1! \alpha_2!)^sh^{|\alpha_1+\alpha_2|}}
\leq \sum_{\gamma\leq \alpha_1}
{\alpha_1 \choose \gamma}
\frac{|x^{\alpha_1-\gamma}
y^{\alpha_2+\gamma}\psi_0(x,y)|}
{(\alpha_1! \alpha_2!)^sh^{|\alpha_1+\alpha_2|}}
\\
= \sum_{\gamma\leq \alpha_1}\frac{\alpha_1!}
{\gamma!(\alpha_1-\gamma)!
(\alpha_1!\alpha_2!)^s}
\frac{|x^{\alpha_1-\gamma}
y^{\alpha_2+\gamma}\psi_0(x,y)|}
{h^{|\alpha_1+\alpha_2|}}
\\
=\sum_{\gamma\leq \alpha_1}
\left( \frac{\alpha_1!}
{(\alpha_1-\gamma)!\gamma!} \right)^{1-s}
\left( \frac{(\alpha_2+\gamma)!}
{\alpha_2!\gamma!} \right)^s
\frac{|x^{\alpha_1-\gamma}
y^{\alpha_2+\gamma}\psi_0(x,y)|}
{((\alpha_1-\gamma)!(\alpha_2+\gamma)!)^s 
h^{|\alpha_1+\alpha_2|}}
\\
\leq 2^{{|\alpha_1|}{(2-s)}} 
2^{|\alpha_1+\alpha_2|s} \sup_{\gamma}
\frac{|x^{\alpha_1-\gamma}
y^{\alpha_2+\gamma}\psi_0(x,y)|}
{((\alpha_1-\gamma)!(\alpha_2+\gamma)!)^s 
h^{|\alpha_1+\alpha_2|}}
\\
\leq 4^{|\alpha_1+\alpha_2|}
\sup_{\gamma}
\frac{|x^{\alpha_1-\gamma}
y^{\alpha_2+\gamma}\psi_0(x,y)|}
{((\alpha_1-\gamma)!(\alpha_2+\gamma)!)^s 
h^{|\alpha_1+\alpha_2|}}.
\end{multline*}
Similarly,
\begin{multline*}
\frac{|D_x^{\beta_1}D_y^{\beta_2}\psi_2(x,y)|}
{(\beta_1! \beta_2!)^sh^{|\beta_1+\beta_2|}}
=\frac{|D_x^{\beta_1}D_y^{\beta_2}\psi_0(x+y,y)|}
{(\beta_1! \beta_2!)^sh^{|\beta_1+\beta_2|}}
\\
=\frac{|D_y^{\beta_2}((D_x^{\beta_1}\psi_0)(x+y,y))|}
{(\beta_1! \beta_2!)^sh^{|\beta_1+\beta_2|}}
\\
\leq \sum_{\delta}{\beta_2 \choose \delta}
\frac{|(D_x^{\beta_1+\delta}D_y^{\beta_2-\delta}
\psi_0)(x+y,y)|}
{(\beta_1! \beta_2!)^sh^{|\beta_1+\beta_2|}}
\\
\leq 4^{|\beta_1+\beta_2|}
\sup_{\delta}
\frac{|D_x^{\beta_1+\delta}
D_y^{\beta_2-\delta}\psi_0(x+y,y)|}
{((\beta_1+\delta)!(\beta_2-\delta)!)^s 
h^{|\beta_1+\beta_2|}}.
\end{multline*}
A combination of these arguments give
\begin{multline*}
\|\psi_2\|^{(1)}_{\maclS_{s,h}}
=\sup_{\alpha_1,\alpha_2,\beta_1,\beta_2}
\sup_{x,y}\frac{|x^{\alpha_1}y^{\alpha_2}
D_x^{\beta_1}D_y^{\beta_2}\psi_2(x,y)|}
{(\alpha_1!\alpha_2!\beta_1!\beta_2!)^s
h^{|\alpha_1+\alpha_2+\beta_1+\beta_2|}}
\\
\leq 4^{|\alpha_1+\alpha_2+\beta_1+\beta_2|}
\sup_{x,y} \frac{|x^{\alpha_1-\gamma}y^{\alpha_2+\gamma}
D_x^{\beta_1+\delta}D_y^{\beta_2-\delta}\psi_0(x,y)|}
{((\alpha_1-\gamma)!(\alpha_2+\gamma)!
(\beta_1+\delta)!(\beta_2-\delta)!)^s 
h^{|\alpha_1+\alpha_2+\beta_1+\beta_2|}}.
\end{multline*}

The other cases follow by repeating 
this argument, and are left
for the reader. The proof is complete. 
\end{proof}

\par

\begin{lemma}\label{twovarJJk}
Let $s\geq 1/2$, $\psi_0\in 
C^{\infty}(\mathbf R^{2d})\cap 
\maclS_s^{\prime}(\mathbf R^{2d})$, and set 
$\psi_1$ and $\psi_2$ be given by $\eqref{psi1psi2}$.
If $k=0,1,2$, then the following 
conditions are equivalent:
\begin{enumerate}
\item $\psi_0 \in \maclS_s(\mathbf R^{2d})$,

\item $\psi_k \in \maclS_s(\mathbf R^{2d})$,

\item for some positive constants $C$, $h_1$ and $h_2$, it holds 
\begin{align*}
|x^{\alpha}D_x^{\beta}\psi_k(x,y)| &\leq C(\alpha!\beta!)
^sh_1^{\alpha+\beta}e^{-h_2|y|^{1/s}},
\\
|\xi^{\alpha}D_{\xi}^{\beta}\widehat{\psi_k}(\xi,\eta)|& 
\leq C(\alpha!\beta!)^s
h_1^{\alpha+\beta}e^{-h_2|\eta|^{1/s}}.
\end{align*}
\end{enumerate}
Furthermore, the mappings which take $\psi_0$ into 
$\psi_1$ or $\psi_2$ are homeomorphism 
on $\maclS_s(\mathbf R^{2d})$.
\end{lemma}

\par

\begin{proof}
The result follows from Corollary 2.5 
in \cite{CCK} and its proof
together with the fact that $\maclS_s$ 
is invariant under pullbacks
of linear bijections.
The details are left for the reader.
\end{proof}

\par

\begin{proof}[Proof of Theorem \ref{thmkernel}]
Let $(\cdot, \cdot)_K$ be the semi-scalar 
product on $\maclD_s(\mathbf R^d)$
given by 
$$
(\varphi, \psi)_K=(K,\psi \otimes \overline{\varphi}),
$$
for every $\varphi$, $\psi \in \maclD_s(\mathbf R^d)$.
Also let $\|\cdot \|_K$ be the corresponding 
semi-norm, i.e, $\|\varphi \|_K$
is defined by
$$
\| \varphi\|_K^2=(\varphi, \varphi)_K=
(K,\varphi \otimes \overline{\varphi}),
$$
when $\varphi \in \maclD_s(\mathbf R^d)$.
Since $\maclD_s(\mathbf R^d)$ is 
dense in $\maclS_s(\mathbf R^d)$,
the result follows if we prove that for every positive $h$, 
there is a positive constant 
$C=C_h$, such that 
\begin{equation}\label{ieqK}
|(\varphi, \psi)_K| \leq C\|\varphi\|_{S_{s,h}}^{(1)}
\|\psi\|_{S_{s,h}}^{(1)},
\end{equation}
when $\varphi$, $\psi \in \maclD_s(\mathbf R^d) 
\cap \maclS_{s,h}(\mathbf R^d)$.

\par

Since $\chi(0) \ne 0$, and  $K_\chi\in \maclS_{s}^{\prime}
 (\mathbf R^{2d} )$, it follows from Theorem 4.1.23 in \cite{KP} that
$1/\chi \in C_s $ near the origin, and 
$\kappa/\chi \in \maclD_s(\mathbf R^d)$ for some 
$\kappa \in \maclD_s(\mathbf R^d)$ 
which is equal to 1 in a neighborhood 
$\Omega_0$ of the origin. Hence
$$
K_{\kappa} (x,y) = \frac{\kappa(x-y)}
{\chi(x-y)} \cdot K_{\chi}(x,y) 
\in \maclS_{s}^{\prime}(\mathbf R^{2d}),
$$
since it is obvious that the map $(\varphi,K)\mapsto
\varphi(x-y)K(x,y)$ is continuous 
from $\maclD_s(\mathbf R^d)\times
\maclS_s^{\prime}(\mathbf R^{2d})$ to 
$\maclS_{s}^{\prime}(\mathbf R^{2d})$ .
Consequently, we may assume that $\chi$ 
in the assumption is equal 
to 1 in a neighborhood $\Omega$ of the origin.

\par 

Take an even and non-negative 
function $\phi \in \maclD_s(\mathbf R^d)$, such that
$\sum_{j}\phi(\cdot-x_j)=1$, for some lattice 
$\{x_j\}_{j\in J} \subset \mathbf R^d$,
and $\operatorname{supp}\phi +
\operatorname{supp}\phi \subset \Omega$. 
By Cauchy-Schwartz inequality we get 
$$
|(\varphi, \psi)_K| \leq \sum_{j,k\in J}
|(\varphi _j, \psi_k)_K| \leq
\sum_{j,k\in J}\|\varphi_j\|_K\|\psi_k\|_K, \qquad 
\varphi,\psi \in \maclD_s(\mathbf R^d),
$$
where 
$$
\varphi_j(x)=\varphi(x)\phi(x-x_j), \quad \psi_j(x)
=\psi(x)\phi(x-x_j).
$$
Then \eqref{ieqK} follows if we prove that for 
every $h>0$, there are $h_1>0$, $C>0$ 
such that
\begin{eqnarray} \label{ieqS}
\|\varphi_j\|_K \leq C \left( \sup_{\alpha,\beta \in N^d} 
\sup_{x\in \mathbf R^d}
\frac{|x^{\alpha}D^{\beta}\varphi(x)|}
{(\alpha! \beta!)^s h^{|\alpha+\beta|}}\right)
\cdot e^{-h_1|x_j|^{1/s}}
\nonumber \\ 
= C\|\varphi \|^{(1)}_{\maclS_{s,h}} \cdot  e^{-h_1|x_j|^{1/s}}.
\end{eqnarray}

\par 

In order to prove \eqref{ieqS}, we note that 
the support of $\varphi_j(\cdot +x_j)$
is contained in $\operatorname{supp} \phi $. This gives
\begin{equation*}
\|\varphi_j\|_K^2=(K_{j,\chi},\varphi_j(\cdot+x_j)
\otimes \overline{\varphi_j(\cdot+x_j)}), 
\end{equation*}
where 
$$
K_{j,\chi}(x,y)=K(x+x_j,y+x_j)\chi(x-y).
$$ 
It follows from the definitions that for every $\varepsilon>0$,
$$
e^{-{\varepsilon}|x_j|^{1/s}}K_{j,\chi} 
$$
is bounded in  $\maclS_s^{\prime}
(\mathbf R^{2d})$ with respect to $j\in J$.
Then for every positive $\varepsilon$ and $h$, 
there is a constant $C_{\varepsilon,h}$
such that
\begin{equation}\label{normphik}
\|\varphi_j\|_K \leq C_{\varepsilon,h} e^{\varepsilon|x_j|^{1/s}/2}
\sup_{\alpha,\beta}\sup_x
\frac{|x^{\alpha}D_x^{\beta}\varphi_j(x+x_j)|}
{(\alpha! \beta!)^s h^{|\alpha+\beta|}}.
\end{equation}

\par

Let $\psi_0(x,y)=\varphi(x)\phi(y)$, and let 
$\psi_1$ and $\psi_2$ be as in Lemma \ref{twovarJJk}.
Then $\psi_2(x,x_j)=\varphi_j(x+x_j)$.
By Lemmas \ref{twoseminorm} and \ref{twovarJJk},
it follows that for every $h>0$, there are constants
$C$, $h_1$, $h_2>0$ such that 
$$
\|\psi_2\|^{(2)}_{\maclS_{s,h_1,h_2}} 
\leq C\|\psi_2\|^{(1)}_{\maclS_{s,4h}}
\leq C\|\psi_0\|^{(1)}_{\maclS_{s,h}}.
$$
Consequently, 
$$
 \sup_{\alpha,\beta_1,\beta_2}
\sup_{x,y}\frac{e^{h_2|y|^{1/s}}|x^{\alpha}D_x^{\beta_1}
D_y^{\beta_2}\psi_2(x,y)|}
{(\alpha!\beta_1!\beta_2!)^sh_1^{|\alpha+\beta_1+\beta_2|}}
\leq C\|\psi_0\|^{(1)}_{\maclS_{s,h}},
$$
giving that
\begin{eqnarray}\label{ineqphi}
|x^{\alpha}D_x^{\beta}\varphi_j(x+x_j)|
&=&|x^{\alpha}D_x^{\beta}\psi_2(x,x_j)|
\nonumber\\
&\leq& C e^{-h_2|x_j|^{1/s}}(\alpha!\beta!)^s
h_1^{|\alpha+\beta|}\|\psi_0\|^{(1)}_{\maclS_{s,h}},
\end{eqnarray}
if we choose $\beta_1=\beta$ and $\beta_2=0$.

\par

Then the
inequalities \eqref{normphik} and \eqref{ineqphi} imply that 
for every $h_3>0$, there is a constant $C_1$ such that 
\begin{eqnarray*}
\|\varphi_j\|_K &\leq& C_1
e^{\varepsilon|x_j|^{1/s}/2}
e^{-h_2|x_j|^{1/s}}
\left(\frac{h_1}{h_3}\right)^{|\alpha+\beta|}
\|\psi_0\|^{(1)}_{\maclS_{s,h}}
\\
&\leq& C_{\phi}e^{-(h_2-\varepsilon/2)|x_j|^{1/s}}
\left(\frac{h_1}{h_3}\right)^{|\alpha+\beta|}
\|\varphi\|^{(1)}_{\maclS_{s,h}},
\end{eqnarray*}
and \eqref{ieqK} follows if we choose 
$h_3$ and $\varepsilon$ such that
 $h_1<h_3$ and $\varepsilon<2h_2$.
\end{proof}

\par

By choosing $K(x,y)=a(x-y)B(x,y)$
in Theorem $\ref{thmkernel}$,
for suitable $a$ and $B$, we get 
the following result, which shows 
(1) in Proposition \ref{DCS} has an analogue
in the framework of Gelfand-Shilov space or
Gevrey class.

\begin{thm}\label{thmgrowh}
Let $s>1$ and $B\in C_s(\mathbf R^{2d})$ be such that
\eqref{ineqB} holds.
Then $\maclD_{B,s,+}^{\prime}(\mathbf R^d)=
\maclS_{B,s,+}^{\prime}(\mathbf R^d)$. In particular, 
$\maclD_{s,+}^{\prime}(\mathbf R^{2d}) = 
\maclS_{s,+}^{\prime}(\mathbf R^{2d})$.
\end{thm}

\par

\begin{proof}
It is sufficient to prove the first inclusion. 
By straight-forward computation
it follows that 
$T\varphi=a*_B\varphi$ and the kernel
$$
K(x,y)=a(x-y)B(x,y)\in
\maclD_{0,s}^{\prime}(\mathbf R^{2d}) \subseteq
\maclS_{0,s}^{\prime}(\mathbf R^{2d}).
$$
Furthermore, Lemma \ref{BCon} giving that 
$$
\frac{1}{B(x,y)}\cdot K(x,y) = a(x-y)\in \maclS_s^{\prime}
{(\mathbf R^{2d})}.
$$
\end{proof}

\par

\section{Gelfand-Shilov properties for positive 
elements for twisted convolutions}\label{sec3}

\par

In this section, 
we study elements in $\mascS _+^\prime $,
which are smooth and have the Gevrey regularity
near the origin,
then they are in $\mathcal S _s$ for
$s\geq 1/2$.
Here $\mascS _+^\prime $
is a set such that for all elements $a\in 
\mascD^{\prime}$ which are positive with 
respect to the twisted convolution.

\vspace{0.5cm}

The following theorem is the main result. It shows that 
(4) in Proposition \ref{DCS} has an analogue in the framework 
of Gelfand-Shilov space or Gevrey class.

\par

\begin{thm}\label{maintheorem}
Let $s\geq 1/2$ and $ a \in \maclS^\prime_{1/2}
 (\mathbf R^{2d})
\cap C^{\infty} (\Omega)$ be such that
\begin{equation} \label{eq.main}
 \sup_{\beta \in \mathbf N^{2d}} \left ( {\sup_{X\in \Omega} \frac {|D^\beta a(X)|}
{(\beta !) ^s h^{|\beta|}}} \right ) < \infty  ,
\end{equation}
for some neighborhood $\Omega $ of the origin, and some $h>0$.
Then $a\in \mathcal S _s(\mathbf R^{2d})$.
\end{thm}

\par

In order to prove the theorem, we need some preparations.

\par

\vspace{0.3cm}

\par

Let
\begin{equation} \label{eq.xx}
\begin{alignedat}{2}
T_j &= (2i) ^{-1}\partial_{\xi_j} - x_j, & \qquad
\Theta_j &= (2i)^{-1}\partial _{x_j} -\xi_j,
\\[1ex]
P_j &= (2i)^{-1}\partial _{\xi_j} + x_j, & \qquad
\Pi_j &= (2i)^{-1}\partial_{x_j} + \xi_j.
\end{alignedat}
\end{equation}
Then 
\begin{equation}  \label{eq.PT}
\begin{alignedat}{2}
x_j &= (P_j-T_j)/2, &\qquad
\xi_j &= (\Pi_j-\Theta_j)/2,
\\[1ex]
\partial_{x_j} &= i(\Pi_j+\Theta_j), &\qquad
\partial_{\xi_j} &= i(P_j+T_j).
\end{alignedat}
\end{equation}
for $j=1,...,d$.

\par

\begin{lemma} \label {Dalpha}
Let $\Omega \subset \mathbf R^{2d}$ be a 
neighborhood of the origin, and let
$a\in C^{\infty}(\Omega)$ and $s\geq 1/2$. Then
the following statements are equivalent. 
\begin{enumerate}
\item There are positive constants $C$ and $h$ such that 
$$
|D^{\alpha}a(x,\xi)| \leq Ch^{|\alpha|}(\alpha!)^s,  
\quad (x,\xi) \in \Omega, 
$$
for the multi-index $\alpha$.

\medspace

\item There are positive constants $C$ and $h$ such that 
$$
|(P^{\alpha}\circ T ^{\beta}\circ 
\Theta ^{\gamma}\circ \Pi ^{\delta}a)(x,\xi)| 
\leq Ch^{|\alpha+\beta+\gamma+\delta|}
(\alpha! \beta! \gamma!
\delta!)^s,   \quad  (x,\xi) \in \Omega, 
$$
for every multi-indices $\alpha$, $\beta$, $\gamma$ and $\delta$.
\end{enumerate}
\end{lemma}

\par

\medspace

\par

\begin{lemma} \label {exchange}
Let $\alpha$, $\beta$, $P_j$, $T_j$, $\Theta_j$, and $\Pi_j$ 
be defined as before, then 
\begin{equation}\label{eq.2.1}
\begin{alignedat}{2}
P ^{\alpha} \circ T ^{\beta} &= T ^{\beta} \circ P ^{\alpha}, 
& \qquad \Pi ^{\alpha}\circ \Theta ^{\beta} &=
\Theta^{\beta}\circ \Pi^{\alpha},
\\[1ex] 
P ^{\alpha}\circ \Pi ^{\beta} &= \Pi ^{\beta}\circ P ^{\alpha}, 
& \qquad T ^{\alpha}\circ \Theta ^{\beta} &=
\Theta^{\beta}\circ T^{\alpha}.
\end{alignedat}
\end{equation}
and
\begin{align}
P ^{\alpha} \circ \Theta ^{\beta} &= 
\sum _{\alpha_0\le \alpha,\beta }i ^{\alpha _0} 
{\alpha \choose {\alpha _0}}{\beta \choose \alpha _0} \alpha_0! 
\Theta^{\beta -\alpha _0} \circ P^{\alpha -\alpha _0},
\label{eq.2.2}
\\[1ex]
\Theta ^{\alpha} \circ P ^{\beta} &=\sum _{\alpha_0\le 
\alpha,\beta }(-i) ^{\alpha _0} {\alpha \choose {\alpha _0}}
{\beta \choose \alpha _0} \alpha_0!\ P ^{\beta -\alpha _0} 
\circ \Theta ^{\alpha -\alpha _0},
\\[1ex]
T ^{\alpha} \circ \Pi ^{\beta} &= \sum _{\alpha_0\le \alpha,
\beta }(-i) ^{\alpha _0} {\alpha \choose {\alpha _0}}
{\beta \choose \alpha _0} \alpha_0!\ \Pi ^{\beta -\alpha _0} 
\circ T ^{\alpha -\alpha _0},
\\[1ex]
\Pi ^{\alpha} \circ T ^{\beta} &= \sum _{\alpha_0\le \alpha,\beta }
i ^{\alpha _0} {\alpha \choose {\alpha _0}}
{\beta \choose \alpha _0} \alpha_0! \ T ^{\beta -\alpha _0} 
\circ \Pi^{\alpha -\alpha _0}.
\end{align}
\end{lemma}

\par

\begin{proof}
The identities \eqref{eq.2.1} follow by straight-forward computations. 
The formula \eqref{eq.2.2} follows if we prove
$$
((P ^{\alpha} \circ \Theta ^{\beta})F)(x,\xi) = 
\sum _{\alpha_0\le \alpha,\beta }
i ^{\alpha _0} {\alpha \choose {\alpha _0}}
{\beta \choose \alpha _0} \alpha_0!
((\Theta ^{\beta-\alpha_0} \circ P ^{\alpha-\alpha_0})F)(x,\xi),
$$
for every $F\in \mascS (\mathbf R^{2d})$.

\par

Let $(\mathscr F_1F)(\eta,\xi)$ be the 
partial Fourier Transform of
$F(x,\xi)$ with respect to the $x$-variable.
Then $P _j$ and $\Theta _j$ are transformed into the operators
$$
\Phi(\eta,\xi) \mapsto \left (\frac{1}{2i}
\frac{\partial}{\partial \xi_j} - \frac{1}{i}
\frac{\partial}{\partial \eta_j} \right ) \Phi(\eta,\xi)
\quad \text{and}\quad
\Phi(\eta,\xi) \mapsto 
\left (\frac{1}{2}\eta_j - \xi_j \right ) \Phi(\eta,\xi) ,
$$
respectively, and by letting 
$$
\sigma_j=\frac{1}{2}\eta_j + \xi_j , \qquad  \tau_j=
\frac{1}{2}\eta_j - \xi_j,
$$
and letting $G$ be defined by $G(\sigma_j,\tau_j) = 
(\mathscr F_1F)(\eta_j,\xi_j)$, it follows that
$P _j$ and $\Theta _j$ are transformed into the operators
$$
G\mapsto -D_{\tau _j}G\quad \text{and}
\quad G\mapsto \tau _jG,
$$
respectively. The relative \eqref{eq.2.2} is 
now a consequence of the Leibniz
rule \eqref{eq.leib}.

\par

The other statements follow by similar arguments, 
and are left for the reader. 
\end{proof}

\par

\begin{lemma} \label{PTR}
Let $a\in C^{\infty}(\Omega)$, 
where $\Omega \subset \mathbf R^{2d}$ is open, 
$s\geq 1/2$, and let $P_j$, $T_j$, $\Theta_j$ and 
$\Pi_j$ be defined as before. 
Also let $R_{\alpha, \beta, \gamma,\delta}$ 
be a composition of $P^{\alpha}$, $T^{\beta}$, 
$\Theta^{\gamma}$ and $\Pi^{\delta}$. 
Then the following statements are equivalent.
\begin{enumerate}
\item There exist positive constants $C$ and $h$ such that 
\begin{equation}\label{compositions}
|((P^{\alpha}\circ T^{\beta}\circ \Theta^{\gamma} 
\circ \Pi^{\delta})a)(x,\xi)| 
\leq C h ^{|\alpha+\beta +\gamma +\delta |}
(\alpha ! \beta !\gamma ! \delta!)^s,
\end{equation}
when $(x,\xi) \in \Omega$, for every multi-indices 
$\alpha$, $\beta$, $\gamma$ and $\delta$.

\item There exist positive constants $C$ and $h$ such that 
$$
|R_{\alpha, \beta, \gamma, \delta}a(x,\xi)| \leq 
C h^{|\alpha +\beta +\gamma +\delta |}(\alpha !
\beta !\gamma ! \delta!)^s,
$$
when $(x,\xi) \in \Omega$, for every multi-indices 
$\alpha$, $\beta$, $\gamma$ and $\delta$.

\item $a\in \maclS _s (\mathbf R^{2d}).$

\end{enumerate}
\end{lemma}

\par

\begin{proof}
We only prove the result for 
$R_{\alpha,\beta,\gamma,\delta}=
\Theta^{\gamma} \circ P^{\alpha} 
\circ T^{\beta} \circ \Pi^{\delta}$. 
The other cases follow by repeating 
these arguments, and are left for the reader. 

First assume that (1) holds, 
and choose $h \geq 1$ and $C>0$ 
such that \eqref{compositions} holds.
By Lemma \ref{exchange} it follows that
$$
R_{\alpha,\beta,\gamma,\delta} = 
\sum _{\alpha_0 \leq \alpha,\delta} 
 (-i)^{\alpha _0} {\alpha \choose \alpha_0}  
{\gamma \choose \alpha_0}
\alpha_0! P^{\alpha-\alpha_0} \circ  
T ^{\beta} \circ \Theta ^{\gamma-\alpha_0}
\circ \Pi^{\delta}.
$$
Hence \eqref{compositions} gives
\begin{multline*}
|(R_{\alpha,\beta,\gamma,\delta}a)(x,\xi)| 
\le \sum _{\alpha_0 \leq \alpha,\delta}
{\alpha \choose \alpha_0}{\gamma \choose \alpha_0}
 \alpha_0! |((P^{\alpha-\alpha_0}\circ 
T^{\beta}\circ \Theta^{\gamma-\alpha_0}
\circ \Pi^{\delta})a)(x,\xi)|
\\[1ex]
\le Ch^{|\alpha+\beta+\gamma+\delta|} 
\left(\alpha!\beta!\gamma!\delta! \right)^s \sum
_{\alpha_0 \leq \alpha,\gamma}  
{\alpha \choose \alpha_0}{\gamma \choose \alpha_0}
\alpha_0! \left(\frac{(\alpha-\alpha_0)!
(\gamma-\alpha_0)!}{\alpha! \gamma!}\right)^s .
\end{multline*}
By Cauchy Schwartz's inequality in 
combination with the fact that $s\ge 1/2$, 
it follows that the terms in the
sum on the right-hand side can be estimated as, 
\begin{multline*}
{\alpha \choose \alpha_0}{\gamma \choose \alpha_0} 
\alpha_0! \left(\frac{(\alpha-\alpha_0)!
(\gamma-\alpha_0)!}{\alpha! \gamma!}\right)^s
=
{\alpha \choose \alpha_0} ^{1-s}
{\gamma \choose \alpha_0} ^{1-s} 
(\alpha _0!)^{1-2s}
\\[1ex]
\le
 {\alpha \choose \alpha_0} ^{1-s}
{\gamma \choose \alpha_0} ^{1-s} \le \frac 12
\left (  {\alpha \choose \alpha_0} ^{2-2s} +
{\gamma \choose \alpha_0} ^{2-2s} \right )
\\[1ex]
\le
\frac 12
\left ({\alpha \choose \alpha_0}+ 
{\gamma \choose \alpha_0} \right ).
\end{multline*}

\par

A combination of these estimates give
\begin{multline*}
|(R_{\alpha,\beta,\gamma,\delta}a)(x,\xi)|
\\[1ex]
\le
2^{-1}Ch^{|\alpha+\beta+\gamma+\delta|} 
\left(\alpha!\beta!\gamma!\delta! \right)^s
\left ( \sum _{\alpha _0\le \alpha} {\alpha \choose \alpha_0}+
\sum _{\alpha _0\le \gamma}{\gamma \choose \alpha_0} 
\right )
\\[1ex]
= Ch^{|\alpha+\beta+\gamma+\delta|} 
\left(\alpha!\beta!\gamma!
\delta! \right)^s(2^{|\alpha |-1}+2^{|\gamma |-1}),
\end{multline*}
which proves statement (2).

\par

Next we prove that (2) gives (3). 
Therefore assume that (2) holds. We have 
$$
|(x^{\alpha_1}\xi^{\alpha_2}D_x^{\beta_1}
D_{\xi}^{\beta_2}a)(x,\xi)|=
|(x^{\alpha_1}\xi^{\alpha_2}
D_{\xi}^{\beta_2}D_x^{\beta_1}a)(x,\xi)|.
$$
By \eqref{eq.exchangorder} we get
\begin{multline} \label{ieq.exchangeorder}
|(x^{\alpha_1}\xi^{\alpha_2}
D_{\xi}^{\beta_2}D_x^{\beta_1}a)(x,\xi)|  
\\[1ex]
\leq \sum_{\alpha_0 \leq \alpha_2,\beta_2} 
{\alpha_2 \choose \alpha_0}
{\beta_2 \choose \alpha_0}\alpha_0!
|(x^{\alpha_1}D_{\xi}^{\beta_2-\alpha_0}
\xi^{\alpha_2-\alpha_0}D_x^{\beta_1}a)(x,\xi)|.
\end{multline}
For the last factor, using \eqref{eq.PT}, we obtain 
\begin{multline*}
|(x^{\alpha_1}D_{\xi}^{\beta_2-\alpha_0}
\xi^{\alpha_2-\alpha_0}D_x^{\beta_1}a)(x,\xi)|
\\[1ex]
=2^{-|\alpha_1+\alpha_2-\alpha_0|}|
(((P-T)^{\alpha_1} (P+T)^{\beta_2-\alpha_0}
(\Pi-\Theta)^{\alpha_2-\alpha_0}
(\Pi+\Theta)^{\beta_1})a)(x,\xi)|.
\end{multline*}
By the binomial theorem and (2) it follows
that the last term can be estimated by 
$$
Ch^{|\alpha_1+\alpha_2+\beta_1+\beta_2-2\alpha_0|}(\alpha_1!
(\alpha_2-\alpha_0)! \beta_1!(\beta_2-\alpha_0)!)^s.
$$
Inserting this into \eqref{ieq.exchangeorder} gives
\begin{multline*}
|(x^{\alpha_1}\xi^{\alpha_2}
D_{\xi}^{\beta_2}D_x^{\beta_1}a)(x,\xi)| 
\\[1ex]
\leq C\sum_{\alpha_0 \leq \alpha_2,\beta_2} 
{\alpha_2 \choose \alpha_0}
{\beta_2 \choose \alpha_0}\alpha_0!
h^{|\alpha_1+\alpha_2+\beta_1+\beta_2-2\alpha_0|}
(\alpha_1! (\alpha_2-\alpha_0)! \beta_1!(\beta_2-\alpha_0)!)^s.
\end{multline*}
Then (3) follows by similar arguments as in the first part of the proof. 

\par

Now we prove that (3) gives (1). If (3) holds, we get 
\begin{equation}\label{eq.ss}
|(x^{\alpha_1} \xi^{\alpha_2} D_x^{\beta_1} 
D_{\xi}^{\beta_2}a)(x,\xi)| 
\leq Ch^{|\alpha_1+\alpha_2+\beta_1+\beta_2|}
(\alpha_1! \alpha_2! \beta_1! \beta_2)^s.
\end{equation}
By \eqref{eq.xx} and the binomial theorem, we obtain that 
\begin{multline*}
|((P^{\alpha}\circ T^{\beta}\circ \Theta^{\gamma} 
\circ \Pi^{\delta})a)(x,\xi)| 
\\[1ex]
=\left|\left(\left(\frac{1}{2i}\partial_{\xi}+ x\right)^{\alpha} 
\left(\frac{1}{2i}\partial_{\xi} - x\right)^{\beta} 
\left(\frac{1}{2i}\partial _{x} -\xi \right)^{\gamma} 
\left(\frac{1}{2i}\partial_{x} + \xi \right)^{\delta}a\right)(x,\xi) \right|
\\[1ex]
\leq \sum_{\alpha_0, \beta_0, \gamma_0, \delta_0}
C_{\alpha_0, \beta_0, 
\gamma_0, \delta_0} |(x^{\alpha+\beta-\alpha_0-\beta_0} 
\partial_{\xi}^{\alpha_0+\beta_0} \xi^{\gamma_0+\delta_0}
\partial_x^{\gamma+\delta-\gamma_0-\delta_0}a)(x,\xi) |,
\end{multline*}
where 
$$
C_{\alpha_0, \beta_0, \gamma_0, \delta_0} =  
{\alpha \choose \alpha_0}{\beta \choose \beta_0}
{\gamma \choose \gamma_0}{\delta \choose \delta_0}.
$$
Here the sum it taken over all $\alpha_0,$ $\beta_0$, 
$\gamma_0$, and $\delta_0$ such that 
$\alpha_0 \leq \alpha$, $\beta_0 \leq \beta$, 
$\gamma_0 \leq \gamma$, and $\delta_0\leq \delta$.

\par

By \eqref{eq.leib}, \eqref{eq.ss} and similar 
arguments as in the first part of the proof we obtain 
$$
|((P^{\alpha}\circ T^{\beta}\circ 
\Theta^{\gamma} \circ \Pi^{\delta})a)(x,\xi)| 
\leq C_1 h^{|\alpha+\beta+\gamma+\delta|} 
\sum_{\alpha_0, \beta_0, \gamma_0, 
\delta_0} D_{\alpha_0, \beta_0, \gamma_0, \delta_0},
$$
where $D_{\alpha_0, \beta_0, \gamma_0, \delta_0}$ is given by 
\begin{multline*}
C_{\alpha_0, \beta_0, \gamma_0, \delta_0} 
((\alpha+\beta-\alpha_0-\beta_0)!
(\alpha_0+\beta_0)!(\gamma+\delta-\gamma_0-\delta_0)!
(\gamma_0+\delta_0)!)^s
\\[1ex] 
=((\alpha+\beta)!(\gamma+\delta)!)^s 
{\alpha \choose \alpha_0}{\beta \choose \beta_0}
{\gamma \choose \gamma_0}{\delta \choose \delta_0}
{\alpha+\beta \choose \alpha_0+
\beta_0}^{-s}{\gamma+\delta \choose \gamma_0+\delta_0}^{-s}
\\[1ex]
\leq ((\alpha+\beta)!(\gamma+\delta)!)^s 
{\alpha \choose \alpha_0}{\beta \choose \beta_0}
{\gamma \choose \gamma_0}{\delta \choose \delta_0}.
\end{multline*}
This gives 
\begin{multline*}
|((P^{\alpha}\circ T^{\beta}\circ 
\Theta^{\gamma} \circ \Pi^{\delta})a)(x,\xi)| 
\\[1ex]
\leq C_1 h^{|\alpha+\beta+\gamma+\delta|} 
\sum_{\alpha_0, \beta_0, 
\gamma_0, \delta_0}((\alpha+\beta)!(\gamma+\delta)!)^s
 C_{\alpha_0, \beta_0, \gamma_0, \delta_0}
\\[1ex]
=C_1 2^{|\alpha+\beta+\gamma+\delta|}
h^{|\alpha+\beta+\gamma+\delta|} 
((\alpha+\beta)!(\gamma+\delta)!)^s.
\end{multline*}
Since $(\alpha+\beta)!\leq 2^{|\alpha+\beta|}\alpha! \beta!$, we get 
$$
|((P^{\alpha}\circ T^{\beta}\circ \Theta^{\gamma} 
\circ \Pi^{\delta})a)(x,\xi)| \leq 
C_1 2^{(|\alpha+\beta+\gamma+\delta|)(s+1)}
h^{|\alpha+\beta+\gamma+\delta|} 
(\alpha! \beta!\gamma!\delta!)^s,
$$
and (1) follows. 
\end{proof}

\par

The next result is closely related to 
Lemma \ref{PTR}, and can be found implicitly in 
\cite{GS}. In order to be self-contained, we 
here give a proof. 

\par

\begin{lemma} \label{onevariable}
Let $\Omega \subset \mathbf R^d$ be open, 
$f \in C^{\infty}(\Omega)$, and $s\geq 1/2$. 
Then the following statements are equivalent. 

\begin{enumerate} 
\item There are positive constants 
$C$ and $h$ such that  \label{st.one}
$$
| x^{\alpha}(D^{\beta} f(x))| \leq C 
h^{|\alpha+\beta|} (\alpha! \beta!)^s, \quad x \in \Omega,
$$
for every multi-indices $\alpha$ and $\beta$.

\medspace 

\item There are positive constants 
$C$ and $h$ such that  \label{st.variable}
$$
|D^{\beta}( x^{\alpha} f(x))| \leq C 
h^{|\alpha+\beta|} (\alpha! \beta!)^s, \quad x \in \Omega,
$$
for every multi-indices $\alpha$ and $\beta$.
\end{enumerate}
\end{lemma}

\par

\begin{proof}
Assume that statement (1) holds. 
By Leibniz rule applied to $D^{\beta}( x^{\alpha} f(x))$ we get
$$
|D^{\beta}( x^{\alpha} f(x))| \leq C 
\sum_{\gamma \leq \alpha, \beta}{\alpha \choose \gamma}
{\beta \choose \gamma}\gamma! 
\left((\alpha-\gamma)!(\beta-\gamma)!\right)^s 
h^{|\alpha+\beta-2\gamma|}
$$
for some constant $C$ which is 
independent of $\alpha$ and $\beta$.
By \eqref{eq.leib}, it now follows by the same argument as 
in the proof of Lemma \ref{PTR} that 

$$
|D^{\beta}( x^{\alpha} f(x))| \leq C  
h^{|\alpha+\beta|} (\alpha! \beta!)^s(2^{|\beta|-1}+2^{|\alpha|-1}),
$$
and the statement (2) follows for some $h\geq1$.

\medspace

Assume instead that (2) holds. 
By \eqref{eq.exchangorder},
then statement (1) follows by similar 
arguments as in the first part of the proof.
\end{proof}

\par

The previous lemma can easily be extended 
to more than one variables. 
The proof is similar to the proof of 
Lemma \ref{PTR} and \ref{onevariable} , and is left fot the reader. 

\begin{lemma} \label{xDR}
Let $R_{\alpha,\beta,\gamma,\delta}$ be 
a composition of the multiplication 
operators $x^{\alpha}$, $\xi^{\beta}$, 
$\partial_{x}^{\gamma}$, and $\partial_{\xi}^{\delta}$. 
Then the following conditions are equivalent. 
\begin{enumerate}
\item There are positive constants $C$ and $h$ such that 
$$
|x^{\alpha}\xi^{\beta}\partial_{x}^{\gamma}
\partial_{\xi}^{\delta}a(x,\xi) | \leq C
h^{|\alpha +\beta +\gamma +\delta |}
(\alpha ! \beta ! \gamma!\delta !)^s, \quad (x,\xi) \in \Omega,
$$
for every multi-indices $\alpha$, $\beta$, $\gamma$ and $\delta$.
\item There are positive constants $C$ and $h$ such that 
$$
R_{\alpha,\beta,\gamma,\delta} \le C h^{|\alpha +\beta +\gamma +\delta |}
(\alpha ! \beta ! \gamma!\delta !)^s,
\quad (x,\xi) \in \Omega,
$$
for every multi-indices $\alpha$, $\beta$, $\gamma$ and $\delta$.
\end{enumerate}
In particular, $a\in \mathcal S_s(\mathbf R^{2d})$ if and only if {\rm{(2)}} holds. 
\end{lemma}

\par

The lemma follows by similar arguments as 
in the proofs of Lemmas \ref{PTR} and 
\ref{onevariable}. The details are left for the reader.

\par 

\medspace

\par 

\begin{proof} [Proof of Lemma \ref{Dalpha}]
Assume that (1) follows, and choose $R=\max(|x|,|\xi|,1)$. Then 
$$
|((P^{\alpha} \circ T^{\beta} \circ \Theta^{\gamma} 
\circ \Pi^{\delta})a)(x,\xi)|
\leq \sum {\alpha \choose \alpha_0}
{\beta \choose \beta_0}{\gamma \choose \gamma_0}
{\delta \choose \delta_0}
Q_{\alpha_0,\beta_0,\gamma_0,\delta_0}(x,\xi),
$$
where 
$$
Q_{\alpha_0,\beta_0,\gamma_0,\delta_0}(x,\xi)=
R^{|\alpha-\alpha_0|+|\beta-\beta_0|+
|\gamma-\gamma_0|+|\delta-\delta_0|} 
|\partial_{\xi}^{\alpha_0}\partial_{\xi}^{\beta_0}
\partial_x^{\gamma_0}\partial_x^{\delta_0}a)(x,\xi)|,
$$
and the sum is taken over all $\alpha_0 \leq \alpha$, 
$\beta_0 \leq \beta$, $\gamma_0 \leq 
\gamma$ and $\delta_0 \leq \delta$. By (1) we have
$$
Q_{\alpha_0,\beta_0,\gamma_0,\delta_0}
(x,\xi)\leq CR^{|\alpha-\alpha_0|+|\beta-\beta_0|+
|\gamma-\gamma_0|+|\delta-\delta_0|} 
h^{|\alpha_0+\beta_0+\gamma_0+\delta_0|}
(\alpha! \beta! \gamma! \delta!)^s.
$$
Since $s\geq 1/2$, by inserting this 
into $Q$, it follows from binomial theorem that 
$$
|((P^{\alpha} \circ T^{\beta} \circ \Theta^{\gamma} 
\circ \Pi^{\delta})a)(x,\xi)|\leq 
C(R+h)^{|\alpha+\beta+\gamma+
\delta|}(\alpha! \beta! \gamma! \delta!)^s.
$$
This gives (2).

If instead (2) holds, then for some 
$\alpha_1$ and $\alpha_2$, \eqref{eq.PT} gives
\begin{multline*}
|D^{\alpha}a(x,\xi)|=|D_{\xi}^{\alpha_1}D_{x}^{\alpha_2}a(x,\xi)|
\\[1ex]
\leq |((P+T)^{\alpha_1}(\Pi+\Theta)^{\alpha_2}a)(x,\xi)|
\\[1ex]
\leq \sum {\alpha_1 \choose \alpha_0}{\alpha_2 \choose \gamma_0}|(P^{\alpha_0} 
\circ T^{\alpha_1-\alpha_0}\circ
\Theta^{\gamma_0}\circ \Pi^{\alpha_2-\gamma_0})a(x,\xi) |.
\end{multline*}
By similar arguments as in the first 
part of the proof as well as in earliar proofs, 
we obtain that the right hand side can be estimated by 
$$
Ch^{|\alpha_1+\alpha_2|}(\alpha_1 ! \alpha_2 !)^s.
$$
This gives the result. 
\end{proof}

\medspace

\par

The next lemma is the last step in the proof of Theorem \ref{maintheorem}.

\par 

\begin{lemma}\label{originSs}
Let $s\geq 1/2$ and $a \in C_+(\mathbf R^{2d})
\cap C^{\prime}(\mathbf R^{2d})$ be such that 
\begin{equation}\label{ineorigin}
|(\partial_x^{\alpha} \partial_{\xi}^{\beta}a)(0,0)|\le 
C h^{|\alpha + \beta|} 
(\alpha! \beta!)^s,
\end{equation}
where $\alpha, \beta
\in \mathbf Z_+^d$, then $a \in \maclS_s(\mathbf R^{2d})$.
\end{lemma}

\par

\begin{proof}
Since $\maclS_s$ is dense in $\mascS$, 
it follows from Theorem 3.3 in \cite{TJ} that $
Aa$ is a positive semi-definite trace-class 
operator on $L^2(\mathbf R^{2d})$. 
In particular, 
$$
(Aa)(x,y)=\sum _jf_j(x)\overline{f_j(y)},
$$
where $(f_j,f_k)=0$ when $j\neq k$, 
and the trace-norm of $Aa$ is given by
\begin{equation}\label{eq.trace}
\|Aa\|_{\operatorname{Tr}}=\sum \|f_j\|_{L^2}^2=
(\pi/2)^{d/2}a(0,0) < \infty.
\end{equation}
More specific, by Theorem 3.13 in \cite{TJ}, it follows that $a\in
\mascS(\mathbf R^{2d})$, and that
$$
\sum  _j \|x^{\alpha}D^{\gamma}f_j\|_{L^2}^2< \infty,
$$
for every multi-indices $\alpha$ and $\gamma$.
Now let $a_{\alpha,\gamma}=P^{\alpha}\circ T^{\alpha} \circ 
\Theta^{\gamma}\circ \Pi^{\gamma}a$. 
Then
$$
Aa_{\alpha,\gamma}=\sum (x^{\alpha}D^{\gamma}f_j)\otimes
(\overline{x^{\alpha}D^{\gamma}f_j}).
$$ 
Furthermore, since $a_{\alpha,\gamma} \in C_+(\mathbf R^{2d})$,
Lemma \ref{Dalpha} gives 
$$
|a_{\alpha,\gamma}(x,\xi)| \leq a_{\alpha,\gamma}(0,0) \leq C
h^{2|\alpha +\gamma|}(\alpha!\gamma!)^{2s},
$$
where $C$ and $h$ are independent of $\alpha$ and $\gamma$.
A combination of these relations and \eqref{eq.trace} give
$$
\|Aa_{\alpha,\gamma}\|_{\operatorname{Tr}}=\sum _j
\|x^{\alpha}D^{\gamma}f_j\|_{L^2}^2
=(\pi/2)^{d/2}a_{\alpha,\gamma}(0,0)
\leq Ch^{|2\alpha +2\gamma|}(\alpha!\gamma!)^{2s},
$$
for some constants $C$ and $h$.

\par

Next let $a_{\alpha,\beta,\gamma,\delta}=
P^{\alpha}\circ T^{\beta} \circ 
\Theta^{\gamma}\circ \Pi^{\delta}a$. 
Then $Aa_{\alpha,\beta,\gamma,\delta}$ is 
a linear combination of terms of the type $\sum
(x^{\alpha}D^{\gamma}f_j)\otimes(\overline{x ^{\beta}D^{\delta}f_j}) $.
By applying Cauchy Schwartz inequality we get
\begin{multline*}
\| Aa_{\alpha,\beta,\gamma,\delta} \|_{\operatorname{Tr}} 
\leq \sum  _j
\|(x^{\alpha}D^{\gamma}f_j)\otimes(\overline{x ^{\beta}
D^{\delta}f_j})\|_{\operatorname{Tr}} 
\\[1ex]
= \sum _j
\|(x^{\alpha}D^{\gamma}f_j)\otimes
(\overline{x ^{\beta}D^{\delta}f_j})\|_{L^2}
\\[1ex]
= \sum _j \|x^{\alpha}D^{\gamma}f_j\|_{L^2}
\|x ^{\beta}D^{\delta}f_j\|_{L^2}
\\[1ex]
\le  \left(\sum _j\|x^{\alpha}D^{\gamma}f_j\|_{L^2}^2\right)^{1/2}
\left(\sum _j \|x ^{\beta}D^{\delta}f_j\|_{L^2}^2\right)^{1/2}
\\[1ex]
\le Ch^{|\alpha+\beta+\gamma+\delta|}
(\alpha !\beta ! \gamma ! \delta !)^s,
\end{multline*}
for some constants $C$ and $h$. 
In the first inequality we have used 
the fact that $(x^{\alpha}D^{\gamma}f_j)
\otimes(\overline{x ^{\beta}D^{\delta}f_j})$
is an operator of rank one. By Proposition \ref{trace}, we get
$$
\|a_{\alpha,\beta,\gamma,\delta}\|_{L^{\infty}} \leq  
Ch^{|\alpha+\beta+\gamma+\delta|}(\alpha !\beta ! \gamma ! \delta !)^s,
$$
which implies that $a\in \mathcal S_s(\mathbf R^{2d})$. 
The proof is complete.
\end{proof}

\par

\begin{proof}[Proof of Theorem \ref{maintheorem}]
By theorem 3.13 in \cite{TJ} it follows that
$a\in \mascS^{\prime}(\mathbf R^{2d})$.
since \eqref{eq.main} implies \eqref{ineorigin},
it follows from Lemma \ref{originSs} 
that $a\in \maclS_s(\mathbf
R^{2d})$, and the result follows. 
\end{proof}

\par

\par

\end{document}